\documentclass[11pt]{article}
\usepackage{latexsym}
\usepackage{amsfonts}
\usepackage{enumerate,graphicx}
\usepackage{verbatim}

\usepackage{amsmath, amssymb, amsthm}

\theoremstyle{definition}

\numberwithin{equation}{section}

\newcommand\vanish[1]{}	% \vanish{TEXT} hides text in the compiled file.

\newcommand\ourcomment[1]{ \textbf{[#1]} }
\newcommand\oc\ourcomment

\newtheorem{theorem}{Theorem}[section]
\newtheorem{conjecture}{Conjecture}[section]

\newtheorem{problem}{Problem}[section]
\newtheorem{proposition}{Proposition}[section]

\title{Cop number of $2K_2$-free graphs}
\author{Vaidy Sivaraman \\ Stephen Testa}

\begin{document}
\maketitle
\begin{abstract}
We prove that the cop number of a $2K_2$-free graph is at most $2$ if it has diameter $3$ or does not have an induced cycle of length $k$, where $k \ \in \{3,4,5\}$. We conjecture that the cop number of every $2K_2$-free graph is at most $2$.
\end{abstract}

\maketitle

 The game of cops and robbers is played on a finite, simple, and connected graph $G$. There are $k$ cops and a single robber. Each of the cops chooses a vertex to start, and then the robber chooses a vertex. And then they alternate turns starting with the cop. In the turn of cops, each cop either stays on the vertex or moves to a neighboring vertex. In the robber's turn, he stays on the same vertex or moves to a neighboring vertex. Each move is seen by both players. The cops win if at any point in time, one of the cops lands on the robber. The robber wins if he can avoid being captured. The cop number of $G$, denoted $c(G)$, is the minimum number of cops needed so that the cops have a winning strategy in $G$.  The question of what makes a graph to have high cop number is not clearly understood. Some fundamental results were proved in \cite{AF} and \cite{NW}. The book by Bonato and Nowakowski \cite{BN} is a fantastic source of information on the game of cops and robbers. (A quick primer on the cop number is \cite{AB}.) \\

All graphs in this article are finite, simple, and connected. For graphs $H,G$ we say that $G$ is $H$-free if $G$ does not contain $H$ as an induced subgraph. A stable set in a graph is a set of pairwise non-adjacent vertices. Let $A,B$ be disjoint vertex sets in $G$. We say that $A$ is complete to $B$ if every vertex in $A$ is adjacent to every vertex in $B$, and that $A$ is anticomplete to $B$ if every vertex in $A$ is non-adjacent to every vertex in $B$. For a positive integer $t$, $P_t$ will denote the path graph on $t$ vertices. A $k$-cycle is a cycle with $k$ vertices (or edges). We denote the complement of the $4$-cycle by $2K_2$. The class of $2K_2$-free graphs has been extensively studied but still no structure theorem is known. The first result proved about ``$\chi$-boundedness" (introduced by Gy\'{a}rf\'{a}s \cite{AG}) is for the class of $2K_2$-free graphs \cite{SW}. \\ 

In this article we are interested in the cop number of $2K_2$-free graphs. We start with a simple proposition.

\begin{proposition}\label{EASYBOUND}
 Let $G$ be a $2K_2$-free graph. Then $c(G) \leq 3$. 
\end{proposition}

\begin{proof}
Here is a winning strategy for 3 cops. Place two stationary cops at two adjacent vertices, say $u$, $v$, and the third cop on $u$. Let $S$ be the set of vertices that is neither a neighbor of $u$ nor a neighbor of $v$. Since $G$ is $2K_2$-free, $S$ is a stable set. In his first move the robber will choose a vertex in $S$, say $w$, but he cannot move from $w$ (since the cops in $u$ and $v$ are not moving). The third cop just walks to $w$, by taking a shortest path from $u$ to $w$, to capture the robber. 
\end{proof}

In the previous simple argument it seems that we are using too may cops. Can we save a cop? We state this as a conjecture. 

 \begin{conjecture}
 Let $G$ be a $2K_2$-free graph. Then $c(G) \leq 2$.
 \end{conjecture}

In this article we prove two partial results towards this conjecture. The first one deals with graphs with diameter $3$. Recall that the diameter of a graph is the maximum length of a shortest path between two vertices.

  \begin{theorem}
  Let $G$ be a $2K_2$-free graph with diameter $3$. Then $c(G) \leq 2$. 
  \end{theorem}
  
  \begin{proof}
Let $v_0$, $v_3$ be vertices such that the distance between them is $3$, and let $v_0-v_1-v_2-v_3$ be a shortest $v_0,v_3$-path. Let $L_i$ denote the set of vertices in $G$ at distance $i$ from $v_0$. Let $B = \{v \in L_2 : v \text{ has a neighbor in } L_3\}$. 
 
 Note that $v_2 \in B$, in particular, $B$ is non-empty. \\
 
 Let $A = L_2 \setminus B$. Let $A_1 = \{v \in A: v \text{ is adjacent to } v_1\}$. Let $A_2 = A \setminus A_1$. Note that $L_4 = \emptyset$ because $G$ is $2K_2$-free. \\
 
 We claim the following. \\

 (1) $L_3$ is stable. \\
 
 For if not, let $a,b \in L_3$ be adjacent vertices. Then $\{a,b,v_0,v_1\}$ induces a $2K_2$, a contradiction.  This proves (1). \\
 
 (2) $B$ is complete to $L_1$. \\
 
 Suppose $b \in B$ is non-adjacent to $b' \in L_1$. Let $b'' \in L_3$ be a neighbor of $b$. Then $\{v_0, b', b, b''\}$ induces a $2K_2$, a contradiction. This proves (2).  \\
 
 (3) $A_2$ is stable. \\ 
 
 Suppose $a_2', a_2'' \in A_2$ be adjacent. Then $\{v_0, v_1, a_2', a_2''\}$ induces a $2K_2$, a contradiction. This proves (3). \\
 
 (4) For any two adjacent vertices $a,b \in A$, at least one of them is adjacent to $v_2$. \\
 
 Suppose neither $a$ nor $b$ is adjacent to $v_2$. Then $\{a,b, v_2, v_3\}$ induces a $2K_2$, a contradiction. This proves (4). \\
 
 We give a winning strategy with 2 cops. We will assume that the robber never places himself in a vertex adjacent to a cop vertex, and that he surrenders when that is not possible. Place one cop at $v_1$ and the other at $v_2$. We claim the following: \\
 
(5) The robber moves to a vertex in $A_2$. \\
 
The robber does not move to $v_0$, $A_1$, $B$ because they are in the neighborhood of $v_1$.  The robber does not move to a vertex $L_1$ because it is in the neighborhood of $v_2$. The robber does not move into $L_3$, because if he does, the cop at $v_1$ stays to block the robber' escape, and the other cop just walks to the robber and catches him. (Note that we are using (2) again here.) This proves (5). \\
 
 Let $z \in A_2$ be the vertex that the robber chooses. This implies that $z$ is non-adjacent to $v_2$. Let $y$ be a neighbor of $z$ in $L_1$. Now, the cop at $v_2$ moves to $y$, and the cop at $v_1$ moves to $v_2$. Now it is the robber's turn, but \\
 
 (6) The robber can only move to a vertex in the neighborhood of either $v_2$ or $y$, and hence surrenders. \\
 
 Since $z$ is a neighbor of $y$, the robber has to move. Since $v_2 \in B$, by using (2) we may assume he does not move to $L_1$ for $L_1$ is in the neighborhood of $v_2$. The robber cannot move to a vertex in $A_2$ since $A_2$ is stable. Since no vertex in $A_2$ is adjacent to a vertex in $L_3$, the robber cannot move to $L_3$. He will not move to a vertex in $B$, $B$ is complete to $y$. Let $w \in A_1$ be a neighbor of $z$. Since $z$ is non-adjacent to $v_2$, by (4) $w$ must be adjacent to $v_2$. Every neighbor of $z$ in $A_1$ is also a neighbor of $v_2$. This proves (6). \\
 
Thus $2$ cops have a winning strategy in $G$, completing the proof of the theorem. 
 \end{proof}

 We remark that we have used the hypothesis that $G$ has diameter $3$ to infer that $L_3$ (as we have defined in the proof) is non-empty, and hence, so is $B$. This was crucial in the proof. We couldn't find a way to adjust the proof so as to work for graphs with diameter $2$. Incidentally, there is a paper \cite{ZAW} on the cop number of diameter $2$ graphs, but the focus there is on Meyniel's conjecture, which states that the maximum cop number of an $n$-vertex graph is of the order of $\sqrt{n}$ (see \cite{BN}).\\
 
 It is tempting to compare the properties ``$\chi$-bounded" and ``bounded cop number". On the positive side, there are graphs with arbitrarily large girth and arbitrarily large cop number (see \cite{BN}). On the negative side, line graphs (a subclass of claw-free graphs) have unbounded cop number (see \cite{JKT}), while they are $\chi$-bounded by a quadratic $\chi$-bounding function. \\
  
Our next result is when the graph under consideration does not have an induced cycle of length $k$, where $k \in \{3,4,5\}$. Note that a $2K_2$-free graph cannot contain an induced $k$-cycle for $k \geq 6$.

\begin{theorem}
Let $G$ be a $2K_2$-free graph. If $G$ has no induced $k$-cycle for some $k \in \{3,4,5\}$, then $c(G) \leq 2$. 
\end{theorem}

\begin{proof}

Suppose that $G$ has no induced $4$-cycle. The main result in \cite{BHPT} is a structure theorem for $(2K_2, C_4)$-free graphs. It states that $V(G)$ is the union of $A,B,C$ where $A$ is either empty or induces a $5$-cycle, $B$ induces a clique and $C$ induces an empty graph (graph with no edges), and $A$ is complete to $B$ and $A$ is anticomplete to $C$. Just placing a cop on a $B$-vertex  gives a winning strategy when $B$ is non-empty. If $B$ is empty, then so is $C$ (because $G$ is connected), and hence $G = C_5$. (Note that $c(C_5) = 2$.) Hence $c(G) \leq 2$.\\

Now suppose that $G$ has no induced $5$-cycle. Let $uv$ be an edge in $G$. Let $A$ be the set of vertices that are neighbors of $u$ but not $v$. Let $B$ be the set of vertices that are neighbors of $v$ but not $u$. Let $C$ be the set of vertices that are neighbors of both $u$ and $v$. Let $D$ be the set of vertices that are neither neighbors of $u$ nor $v$. Since $G$ is $2K_2$-free, $D$ is a stable set. Here is a winning strategy for 2 cops. Place the two cops on $u,v$. The robber will go to a vertex $z \in D$, for otherwise, he will captured in the next move. Suppose that $z$ has a neighbor in $A \cup B$, without loss of generality say $z$ is adjacent to $y \in B$. The cop at $v$ moves to a vertex $y$. The cop at $u$ moves to $v$. Since $G$ is $C_5$-free, any vertex in $A$ that is adjacent to $z$ is also adjacent to $y$. The robber has no move, and is captured in the next move. Suppose that $z$ has no neighbor in $A \cup B$. The cop at $u$ remains stationary, and the cop at $v$ just moves and catches the robber who is unable to move from $z$. Hence $c(G) \leq 2$. \\

Now suppose that $G$ is triangle-free. If $G$ has no induced $5$-cycle, we are done by the previous paragraph. Hence we may assume that $G$ contains an induced $5$-cycle. It is known that every $(2K_2, C_3)$-free graph containing $C_5$ is a blow-up of $C_5$, meaning every vertex has become a stable set (see Theorem $2$ in \cite{TK}). The cop number of a blow-up of $C_5$ is $2$, as the reader can easily check. We conclude that $c(G) \leq 2$. 
\end{proof}
 
 The following stronger conjecture was posed in \cite{VS}. It is stronger because the class of $2K_2$-free graphs is a proper subclass of the class of $P_5$-free graphs. 

  \begin{conjecture}
Let $G$ be a $P_5$-free graph. Then $c(G) \leq 2$.
  \end{conjecture}
  
It is natural to ask about $mK_2$-free graphs, where $mK_2$ denotes the disjoint union of $m$ copies of $K_2$. It is easy to see that the cop number of an $mK_2$-free graph is at most $2m-1$ (by the same argument as in Proposition \ref{EASYBOUND}), but we believe it can be improved. We record it as a problem. 

\begin{problem}
What is the maximum cop number of $mK_2$-free graphs? 
\end{problem}

\end{document}